\documentclass{svmult}
\usepackage{mathptmx}       
\usepackage{amsmath}
\usepackage{amsfonts}
\usepackage{amssymb}
\usepackage{amscd}

\newcommand\g{\gamma}
\renewcommand\d{\delta}

\renewcommand\l{\lambda}
\renewcommand\L{\Lambda}
\newcommand\G{\Gamma}
\newcommand\f{\frac}

\newcommand{\Z}{{\mathbb{Z}}}

\newcommand{\C}{{\mathbb{C}}}

\newcommand{\Q}{{\mathbb{Q}}}

\renewcommand\i{^{-1}}
\renewcommand\({\left(}
\renewcommand\){\right)}

\newcommand{\ignore}[1]{}
\newcommand{\myignore}[1]{}
\newcommand{\mymyignore}[1]{}
\newcommand{\bx}{\hfill$\square$\vspace{.6cm}}

\makeatletter
\def\imod#1{\allowbreak\mkern5mu({\operator@font mod}\,#1)}
\makeatother
\newcommand\smallf[2]{{\textstyle{\frac{#1}{#2}}}}
\newcommand\srel[2]{\begin{smallmatrix} {#1} \\ {#2} \end{smallmatrix}}

\makeatletter
\def\ds@numart{\@numarttrue
  \@takefromreset{figure}{chapter}%
  \@takefromreset{table}{chapter}%
  \@takefromreset{equation}{chapter}%
  \def\thesection{\@arabic\c@section}%
  \def\thefigure{\@arabic\c@figure}%
  \def\thetable{\@arabic\c@table}%
  \def\theequation{\thesection.\arabic{equation}}%
  \def\thesubequation{\arabic{equation}\alph{subequation}}
  \@addtoreset{equation}{section}}
\def\ds@book{\@numartfalse
\def\thesection{\thechapter.\@arabic\c@section}%
\def\thefigure{\thechapter.\@arabic\c@figure}%
\def\thetable{\thechapter.\@arabic\c@table}%
\def\theequation{\thesection.\arabic{equation}}%
\@addtoreset{section}{chapter}%
\@addtoreset{figure}{chapter}%
\@addtoreset{table}{chapter}%
\@addtoreset{equation}{section}}
\makeatother

\newcommand{\thmref}[1]{Theorem~\ref{#1}}

\begin{document}

\title*{Estimates on Eisenstein distributions for reciprocals of $p$-adic $L$-functions:~the case of irregular primes}

\author{Stephen Gelbart, Ralph Greenberg, Stephen D. Miller, and Freydoon Shahidi}

\institute{Stephen Gelbart \at Nicki and J. Ira Harris Professorial Chair, Department of Mathematics,
Ziskind Building, Room 256, Weizmann Institute of Science, Rehovot 76100 Israel, \email{steve.gelbart@weizmann.ac.il}
\and Ralph Greenberg \at Department of Mathematics, Box 354350,
University of Washington,
Seattle, WA 98195-4350, \email{greenber@math.washington.edu}
\newline{Greenberg is supported by NSF grant DMS-0200785.}
\and Stephen D. Miller \at Department of Mathematics,
Hill Center-Busch Campus,
Rutgers, The State University of New Jersey,
110 Frelinghuysen Rd,
Piscataway, NJ 08854-8019, \email{miller@math.rutgers.edu}
 \newline Miller is  supported
by NSF grant DMS-1500562.
\and
Freydoon Shahidi \at Department of Mathematics, Purdue University,
150 N. University Street, West Lafayette, IN 47907-2067, \email{shahidi@math.purdue.edu} \newline Shahidi is partially supported
by NSF grant  DMS-1500759.
}

\titlerunning{Estimates on Eisenstein distributions for reciprocals of $p$-adic $L$-functions}

\maketitle

\date{May 9, 2016}

\abstract{
We consider the $p$-adic distributions derived from Eisenstein series in \cite{GMPS}, whose Mellin transforms are reciprocals of the Kubota-Leopoldt  $p$-adic $L$-function.  These distributions were  shown there to be measures when $p$ is regular.  They fail to be measures when $p$ is irregular; in this paper we give quantitative estimates that describe their behavior more precisely.
\newline
\newline    
{\it To Roger Howe on the occasion of his 70th birthday}
}


\section{Introduction}
\label{sec:intro}

We first begin with a brief, elementary description of the content of this paper.  Let $\mu(\cdot)$ be the M\"obius $\mu$-function\footnote{$\mu(n)$ vanishes unless $n$ is squarefree, in which case it is $(-1)^{\omega(n)}$, where $\omega(n)$ is the number of prime divisors of $n$.}, and consider the sum
\begin{equation}\label{lowbrow}
 \pi^k\,\sum_{\srel{n\,\neq\,0}{p\nmid n}}\f{\mu(|n|)}{n^k}
    \sum_{j\,=\,0}^m p^{j(k-1)-mk}e^{2\pi i \bar{n} p^j b/p^m}
\end{equation}
for integers
 $k\ge 2$ and $m\ge 1$; here $\bar{n}$ denotes the modular inverse of $n\imod{p^m}$.  The sum (\ref{lowbrow}) actually defines a rational number, and appears in a $p$-adic distribution formed from Eisenstein series   in \cite{GMPS} (see (\ref{GMPSmeasure}) below).  The present paper gives a characterization of how many times $p$ divides the denominator of the rational number (\ref{lowbrow}).

The goal of \cite{GMPS} was to construct $p$-adic $L$-functions from Fourier coefficients of Eisenstein series, in which sums like (\ref{lowbrow}) occur.  Before describing this, let us first recall that
the values of the Riemann $\zeta$-function at negative odd integers are related to the classical Bernoulli numbers $B_k$ by the formula
\begin{equation}\label{zeta1-kandBk}
    \zeta(1-k) \ \ = \ \ (-1)^{k-1}\,\f{B_{k}}{k}\, , \ \ k\,=\,1,2,3,\ldots.
\end{equation}
The Bernoulli numbers $B_k$ are the values at $x=0$ of the Bernoulli polynomials $B_k(x)$, which are given by the formulas $B_0(x)=1$, $B_1(x)=x-\f 12$, and
\begin{equation}\label{Bkxdef}
    B_k(x) \ \ = \ \ -\f{k!}{(2\pi i)^k}\sum_{n\,\neq\,0}\f{e^{2\pi i n x}}{n^k} \ \ \ \text{for} \ k\,\ge\,2 \ \text{and} \ \ 0 \le x \le 1\,.
\end{equation}
For a prime number $p$, define a distribution\footnote{A $p$-adic distribution $\nu$ on $\Z_p^*$ is a finitely-additive, $\Q_p$-valued function on sets of the form  $b+p^m\Z_p$, $m\ge 1$; $\nu$ is furthermore a {\em $p$-adic measure} if $\nu$'s values are bounded under the $p$-adic valuation.} on $\Z_p^*$ by the formula $\mu_{B,k}(b+p^m\Z_p)=p^{m(k-1)}B_k(\f{b}{p^m})$ (for $0 \le b < p^m$ and $k\ge 1$).  For $c$ coprime to $p$ and $k\ge 1$, define the $p$-adic measure
\begin{equation}\label{Mazurmeasuredef}
    \mu_{k,c}(U) \ \ := \ \ \f{\mu_{B,k}(U) \ - \ c^{-k}\,\mu_{B,k}(cU)}{-k}
\end{equation}
on compact open subsets $U\subset \Z_p^*$.

Throughout this  paper we have fixed embeddings of $\overline{\Q}$ into $\C$ and $\overline{\Q}_p$, so that algebraic numbers in $\C$ can then be regarded as being in $\overline{\Q}_p$.  Consider a Dirichlet character $\chi$ whose conductor is a power of $p$, and regard its values as roots of unity lying in $\overline{\Q}_p$.  Then $\mu_{k,c}$
 has the property that
\begin{equation}\label{Mazurintegral}
    \int_{\Z_p^*}\chi\,d\mu_{k,c} \ \ = \ \ (1\,-\,\chi(c)^{-1})\,(1\,-\,\chi(p)p^{k-1})\,L(1-k,\chi)\,,
\end{equation}
where on the righthand side $\chi$ is thought of as a complex-valued character and the $L$-value is a complex algebraic number (which, when regarded as an element of $\overline{\Q}_p$, is equal to the lefthand side).\footnote{The middle factor $1-\chi(p)p^{k-1}$ in (\ref{Mazurintegral}) is of course 1 unless $\chi$ is the trivial character, owing to the vanishing of $\chi(p)$.  The righthand side is well-defined independent of what convention is used for the value of the trivial character $\imod {p^m}$ at $p$, i.e., whether  it is treated as a primitive character or not. }
That is, the $p$-adic Mellin transform of $\mu_{k,c}=kx^{k-1}\mu_{1,c}$ is the Kubota-Leopoldt $p$-adic $L$-function, a perspective which explains many     number-theoretic properties of both objects.

The measure $\mu_{k,c}$ is a $p$-adic analog of Jacobi's $\theta$-function, whose Mellin transform was used by Riemann to uncover many analytic properties of his eponymous $\zeta$-function, in particular its meromorphic continuation and functional equation.  The paper \cite{GMPS} considered the $p$-adic analog of a completely different derivation of the analytic properties of $\zeta(s)$:~the Langlands-Shahidi method, which involves an analysis of  Fourier expansions of Eisenstein series.  The main result there is an explicit distribution on $\Z_p^*$ given by the formula\footnote{This follows from formulas (3.17) and (3.19) in \cite{GMPS}; note the misprint in formula (3.20) there, which does not include the first minus sign.}
\begin{equation}\label{GMPSmeasure}
    \mu^*_k(b+p^m\Z_p) \ \ = \ \ \f{(-2\pi i)^k}{4\G(k)}\sum_{\srel{n\,\neq\,0}{p\nmid n}}\f{\mu(|n|)}{n^k}
    \(
    \sum_{j\,=\,0}^m p^{j(k-1)-mk}e^{2\pi i \bar{n} p^j b/p^m}-\f{p^{-m}}{1-p^{1-k}}
    \)\,,
\end{equation}
where  $\bar n$ is the modular inverse of $n\imod {p^m}$.
This formula is equal to a constant multiple of Haar measure on $\Z_p^*$ plus a distribution constructed directly from Fourier coefficients of weight-$k$ Eisenstein series.  Its $p$-adic Mellin transform is
\begin{equation}\label{GMPSMellin}
    \int_{\Z_p^*}\chi\i\,d\mu_k^* \ \ = \ \ \f{1}{(1-\chi(p)p^{k-1})\,L(1-k,\chi)}
\end{equation}
for any Dirichlet character $\chi$ whose conductor is a power of $p$, and any positive integer $k$ having the same parity as $\chi$ (this last assumption is necessary to ensure the $L$-value in the denominator is not zero).\footnote{For convergence reasons related to Eisenstein series, $k$ must be at least 3 in \cite{GMPS}; however, all results here are valid in the wider range $k\ge 1$.}

Since the Dirichlet $L$-function  appears in the numerator of  (\ref{Mazurintegral}) but in the denominator of (\ref{GMPSMellin}), the Mellin transforms of $\mu_{k,c}$ and $\mu_k^*$ are nearly reciprocal. In fact, up to some slight modifications involving $c$ they represent inverses of each other in the Iwasawa algebra.  It thus makes sense to study the divisibility properties of the values of  $\mu_k^*$; as the latter are given by Bernoulli numbers, the values of $\mu_k^*$ can be naturally thought of as ``inverse Bernoulli numbers''.

Using the connection with the Iwasawa algebra, the distribution  $\mu^*_k$ was in fact shown to be a measure for {\em regular}\footnote{A prime $p$ is regular if it divides the class number of $\Q(e^{2\pi i/p})$; the first few irregular primes are 37, 59, 67, 101,\ldots.} primes $p>2$ in \cite{GMPS}.  It is also not hard to see that $\mu^*_k$ fails to be a measure for irregular primes $p$.  Our main result quantifies this failure in a sharp way.

\begin{theorem}
\label{mainthm}
Let $p$ be an irregular prime and $k\ge 1$ a fixed integer.  Then there exists constants $c_1, c_2 >0$ such that
\begin{equation}\label{maininequality}
    c_1 \,p^m \ \ \le \ \ \max_{b\,\in\,(\Z/p^m\Z)^*}|\mu^*_k(b+p^m\Z_p)|_p \ \ \le \ \ c_2\,p^m
\end{equation}
for all integers $m\ge 1$.  Put differently, $\max_{b\,\in\,(\Z/p^m\Z)^*}|\mu^*_k(b+p^m\Z_p)|_p=p^{m+O(1)}$ for fixed $p$ and $k$.
\end{theorem}

\medskip

{\bf Remark:}  The inequality in Theorem~\ref{mainthm} is nearly always a sharp equality, and in fact we expect
\begin{equation}\label{expect}
  |\mu^*_k(b+p^m\Z_p)|_p \ \ = \ \ p^m
\end{equation}
for typical choices of $p$, $m$, $k$, and $b$ with extremely high probability.  This can be seen using the formula
\begin{equation}\label{padicmellinversion}
  \mu_k^*(b+p^m\Z_p) \ \ = \ \ \f{1}{\phi(p^m)}\, \sum_{\srel{\chi\imod{p^m}}{\chi(-1)=(-1)^k}}\f{\chi(b)}{(1-\chi(p)p^{k-1})\,L(1-k,\chi)}\,,
\end{equation}
which results from applying $p$-adic Mellin inversion  to (\ref{GMPSMellin}).  When one particular $L$-value has a smaller valuation than the others, its term  dominates the sum and thus $|\mu_k^*(b+p^m\Z_p)|_p=\f{p^{m-1}}{|L(1-k,\chi)|_p}$; in fact such a dominant term comes from a character $\chi$ of order dividing $p-1$.  Thus the closest distance $d_p(k)$ of $1-k$   to a $p$-adic zero of the Kubota-Leopoldt $L$-function influences this valuation.
  If indeed a single term in (\ref{padicmellinversion}) dominates, then one obtains the following precise equality
\begin{equation}\label{sharperexpect}
\big|  \mu_k^*(b+p^m\Z_p)\big|_p  \ \ = \ \    \frac{p^{m-1}}{d_p(k)} 
\end{equation}
for all $b, m $, and $k$.  One situation in which a single term dominates is when there is exactly one index of irregularity for $p$  and the corresponding $\lambda$-invariant is equal to 1.  This is a quite common occurrence as one sees from the extensive calculations in \cite{BH}.  Those calculations show that for $p < 163,000,000$  and  for each index of irregularity, the corresponding $\lambda$-invariant for that index is indeed equal to 1.  However, there may conceivably be examples to the contrary.   When there are two or more indices of irregularity,    or one index with corresponding $\lambda$-invariant $\ge 2$,     the situation becomes more subtle because two or more terms in (\ref{padicmellinversion}) may have the same dominant valuation, and then  $\big|  \mu_k^*(b+p^m\Z_p)\big|_p$ could vary in size for different choices of $b$ depending on whether or not cancellations occur.  However, our limited numerical calculations indicate that the maximum in (\ref{maininequality})  is typically well-behaved.

\medskip

Section \ref{sec:inverseidem} contains some pertinent background on the Iwasawa algebra.  The proofs of the upper and lower bounds in \thmref{mainthm} are given in Sections~\ref{sec:upperbd} and \ref{sec:lowerbd}, respectively.  The main tools needed for these bounds are $p$-adic estimates of Dirichlet $L$-functions, as well as the Ferrero-Washington Theorem \cite{FerrWash}.

\medskip

The authors would like to thank Larry Washington for his helpful discussions.

\section{The inverse measure in terms of idempotents and the Iwasawa algebra}\label{sec:inverseidem}

For our purposes it is convenient to begin with an explanation of the algebraic derivation of the reciprocal measure different than the one given in \cite[Section 4]{GMPS}.  

\subsection{Valuations on $\overline{\Q}_p$}\label{sec:sub:valuations}

Fix a prime $p$ and let $|\cdot|_p$ denote the $p$-adic valuation on $\Q_p$; since $p$ is fixed we shall often drop the subscript from the valuation.  We have that $\Z_p=\{x\in \Q_p||x|\le 1\}$ is the ring of integers of $\Q_p$, and $p\Z_p=\{x\in \Q_p||x|<1\}$ is the maximal ideal of $\Z_p$.  If $x$ is an element of a finite extension $F$ of $\Q_p$, its valuation   will be given by the formula
\begin{equation}\label{valuation}
    |x| \ \ = \ \ |N_{F/\Q_p}(x)|^{1/[F:\Q_p]}\,,
\end{equation}
where $N_{F/\Q_p}(x)=\prod_{\sigma}\sigma(x)$, $\sigma$ running over the $\Q_p$-linear embeddings of $F$ into $\overline{\Q}_p$ (this set coincides with $\operatorname{Gal}(F/\Q_p)$ when $F/\Q_p$ is a Galois extension).  This formula gives an extension of $|\cdot|$ from $\Q_p$ to its algebraic closure $\overline{\Q}_p$.  Of particular interest will be the cyclotomic extensions $\Q_p(\zeta)$, where $\zeta$ is a primitive $p^t$-th root of unity for some integer $t\ge 1$.  The ring of integers of $\Q_p(\zeta)$ is given by
$\{|x|\le 1\}$; its maximal ideal $\{|x|<1\}$   is generated by $\zeta-1$ \cite[p.104]{Robert}.

For any element $x\in \Z_p^*=\{x\in \Q_p||x|=1\}$, the sequence  $(x^{p^n})$ converges to a limit $\omega(x)\in \Z_p^*$; $\omega$ is called the {\em Teichm\"uller} character of $\Z_p^*$ and its image is precisely the set $\Delta$ of all $(p-1)$-st roots of unity in $\Z_p^*$.  Using the factorization $x=\omega(x)\cdot ( \omega(x)\i x)$ one sees the direct product decomposition $\Z_p^*=\Delta\times\Gamma$, where $\Gamma=1+p\Z_p$ is isomorphic to $\Z_p$.  Fix a topological generator  $\g$   of $\G$, so that this isomorphism from $\Z_p$ to $\Gamma$ can be written as $x\mapsto \g^x$.  Any   character $\chi$ of $\Z_p^*$ can be uniquely factored as $\chi=\omega^i \chi_{_\Gamma}$, where $0\le i<p-1$ and $\chi_{_\Gamma}$ is the restriction of $\chi$ to $\Gamma$.  If $\chi$ has finite order, then $\chi_{_\Gamma}$ takes values in $\Z_p[\zeta]$ for some $p^t$-th root of unity $\zeta$; since $\Delta$ is contained in $\Z_p$, all values of $\chi$ in fact lie in $\Z_p[\zeta]$.
The Kubota-Leopoldt $p$-adic $L$-function is defined by the formula
\begin{equation}\label{RGKubotaL}
 L_p(1-k,  \omega^t \phi)  \ \  = \ \  (1- \omega^{t-k}\phi(p)p^{k-1}) \, L(1-k, \omega^{t-k}\phi)
\end{equation}
  for all finite order characters $\phi$ of $\Gamma$, $k\in \Z_{\ge 1}$,  and all integers   $0 \le t \le p-2$.

\subsection{Idempotents for a finite group}\label{sec:sub:idempotents}

Let $G$ and $G'$ denote finite abelian groups and let $\rho\in \operatorname{Hom}(G,G')$.  Then $\rho$ extends to a homomorphism of group rings $\overline{\Q}_p[G]\rightarrow \overline{\Q}_p[G']$ via the formula
\begin{equation}\label{extendrho}
    \rho\(\sum_{g\in G}a_g\, g\) \ \ = \ \ \sum_{g\in G}a_g \,\rho(g)\,.
\end{equation}
In the special case that $G'=\overline{\Q}_p^*$ and $\chi\in \widehat{G}=\operatorname{Hom}(G,\overline{\Q}_p^*)$, this specializes to the   homomorphism
\begin{equation}\label{extendchi}
    \chi\(\sum_{g\in G}a_g \,g\) \ \ = \ \ \sum_{g\in G}a_g\, \chi(g)
\end{equation}
from $\overline{\Q}_p[G]$ to $\overline{\Q}_p$.

It is an immediate consequence of the orthogonality of characters that an element $\theta\in \overline{\Q}_p[G]$ is determined by the values of $\chi(\theta)$, where $\chi$ varies over all elements of $\widehat{G}$.  This can be made explicit using the {\em idempotent}
\begin{equation}\label{echidef}
    e_\chi \ \ := \ \ \smallf{1}{\# G}\sum_{g\in G}\chi^{-1}(g)g \ \ \in \ \ \overline{\Q}_p[G]
\end{equation}
 associated to $\chi\in \widehat{G}$; $\sum_{\chi\in \widehat{G}}e_\chi$ is the identity element $1\cdot id_G$ in the group ring $\overline{\Q}_p[G]$, where $id_G$ is the identity element of $G$.  Then $\overline{\Q}_p[G]$ is equal to the direct sum of its ideals
\begin{equation}\label{directsum}
    \overline{\Q}_p[G] \ \ = \ \ \bigoplus_{\chi\in \widehat{G}}e_\chi \overline{\Q}_p[G]\,,
\end{equation}
and an element $\theta\in  \overline{\Q}_p[G]$ is characterized by the vector $(\chi(\theta))_{\chi\in\widehat{G}}$.
 In more detail,
   the projection of $\theta$ to $e_\chi \overline{\Q}_p[G]$ is  equal to
$
e_{\chi} \theta  =   \chi(\theta)e_{\chi}
$
and $\theta$ is the sum of those projections:
\begin{equation}\label{explicitdirectsum}
\theta ~=~   \sum_{\chi\in\widehat{G}} ~  e_{\chi}\theta  ~=~   \sum_{\chi\in\widehat{G}} ~~\chi(\theta)\bigg(  \frac{1}{\# G}   \sum_{g\in G} ~ \chi^{-1}(g) g \bigg)~=~   \sum_{g \in G}  a_g\, g\,,
\end{equation}
where
$
a_g  =     \frac{1}{\# G}   \sum_{\chi\in\widehat{G}}  \chi^{-1}(g)\chi(\theta)
$.

Let us now consider the embedding
\begin{equation}\label{QpinQpbar}
    \Q_p[G] \ \ \hookrightarrow \ \ \overline{\Q}_p[G]\,,
\end{equation}
and consider $\operatorname{Gal}(\Q_p(\chi)/\Q_p)$, where $\Q_p(\chi)$ is the finite extension of $\Q_p$ obtained by adjoining all values of the character $\chi$.  We use the notation $\d(\chi)=\d\circ \chi\in \widehat{G}$ to denote the Galois conjugate of $\chi$ by an element $\d\in \operatorname{Gal}(\Q_p(\chi)/\Q_p)$.  The image of $\Q_p[G]$ under (\ref{QpinQpbar}) is simple to characterize:
\begin{equation}\label{Galoiscrit}
    \theta\,\in\,\Q_p[G] \ \Longleftrightarrow \ \phi(e_\chi  \theta)\,=\,e_{\phi\circ\chi}  \theta,
    \forall \chi\in\widehat{G} \text{~and~}\phi\in\operatorname{Gal}(\overline{\Q}_p/\Q_p),
\end{equation}
that is, the $e_\chi \theta$ form a Galois-invariant set.
Any element $\theta=\sum_{g\in G}a_g g \in \Q_p[G]$ defines a $\Q_p$-valued measure on $G$ by the formula $\mu_\theta(\{g\})=a_g$, so that the notation
\begin{equation}\label{integrationashomomorphism}
    \chi( \theta) \ \ = \ \ \int_G\chi \,d\mu_\theta
\end{equation}
can equally well be used for (\ref{extendchi}).

Next consider a finite abelian group $G'$ with a surjective  homomorphism $\rho:G'\rightarrow G$, which induces an injective homomorphism from $\widehat{G}$ to $\widehat{G'}$ given by $\chi\mapsto \chi\circ \rho$. For any element $\theta'=\sum_{g'\in G'}a_{g'}g'\in \Q_p[G']$, its image   $\rho(\theta')\in \Q_p[G]$ defined by (\ref{extendrho}) is completely determined by the values $\chi(\theta')$, where $\chi$ ranges over  the image of this induced injective homomorphism $\widehat{G}\hookrightarrow \widehat{G'}$.

The family of groups $G_n=\Z_p^*/(1+p^n\Z_p)\cong (\Z/p^n\Z)^*$, $n\ge 1$, have surjective homomorphisms $\rho_{n}:G_{n+1}\twoheadrightarrow G_n$  and inverse limit  $G_\infty \cong \Z_p^*$.  Likewise, a $\Q_p$-valued distribution on $\Z_p^*$ is an element of the inverse limit of the group rings $\Q_p[G_n]$ with respect to the projections $\Q_p[G_{n+1}]\twoheadrightarrow \Q_p[G_n]$.  An element of this inverse limit is a compatible sequence $\Theta=(\theta_n)$ of $\Q_p$-valued measures on the respective finite abelian groups $G_n\cong (\Z/p^n\Z)^*$.  These compatible measures themselves form a finite-additive measure on $\Z_p^*$, and are determined by the values of $e_\chi  \theta_n$, for all $\chi\in \widehat{G_n}$ and all $n\ge 1$.  Conversely, a Galois-compatible sequence of values (possibly in $\overline{\Q}_p$) for $e_\chi  \theta_n$ in the sense of (\ref{Galoiscrit}) determine a $\Q_p$-valued distribution on $G_\infty \cong \Z_p^*$.
We conclude this subsection with three important examples.

\noindent {\bf Example 1: the normalized Haar distribution $\nu_{\text{Haar}}$.}

Like all locally compact groups, $\Z_p^*$ has a translation-invariant Haar ``measure'' $\nu_{\text{Haar}}$;  up to scaling, it is given by $\nu_{\text{Haar}}(b+p^m\Z_p)=p^{-m}$ for $m\ge 1$.  However, the quotation marks are necessary because Haar measures are {\em a priori} real-valued, not $\Q_p$-valued:~$\nu_{\text{Haar}}(b+p^m\Z_p)$ is unbounded, inconsistent with the use of the term ``$p$-adic measure''.  The terminology notwithstanding,  $\nu_{\text{Haar}}$  is merely a $p$-adic distribution on $\Z_p^*$.
It corresponds to a sequence $(\theta_n)$ with
\begin{equation}\label{example1A}
e_\chi \theta_n \ \ = \ \ \left\{
  \begin{array}{ll}
    0, & \chi\ \text{nontrivial}\,; \\
    \f{p-1}{p}\,e_\chi, & \chi\ \text{trivial}\,.
  \end{array}
\right.
\end{equation}

\noindent {\bf Example 2:~Mazur's measure $\mu_{k,c}$ in (\ref{Mazurmeasuredef}).}

According to  (\ref{Mazurintegral}), $\mu_{k,c}$ corresponds to a sequence
$(\theta_n)$ with
\begin{equation}\label{example2A}
e_\chi  \theta_n \ \ = \ \ (1\,-\,\chi(c)^{-1})\,(1\,-\,\chi(p)p^{k-1})\,L(1-k,\chi)\,e_\chi.
\end{equation}
Unlike the previous example, here $(\theta_n)$ corresponds to a $p$-adic (i.e., bounded) measure.  Because of the rationality in definition (\ref{Mazurmeasuredef}) inherited from the Bernoulli polynomials $B_k(x)$, the Galois-compatibility property (\ref{Galoiscrit}) holds.

\noindent {\bf Example 3:~the distribution $\mu_k^*$ from \cite{GMPS}  in  (\ref{GMPSmeasure}).}

By (\ref{GMPSMellin}), this corresponds to a sequence
$(\theta_n)$ with
\begin{equation}\label{example3A}
    e_\chi  \theta_n \ \ = \ \ \bigg( (1-\chi(p)^{-1}p^{k-1})L(1-k, \chi^{-1})  \bigg)^{-1}\!\!e_\chi\,.
\end{equation}
The Galois-compatibility property (\ref{Galoiscrit}) again holds for $\theta_n$.  This can be seen by the reciprocal relationship between examples 2 and 3 in the ring of fractions in the   Iwasawa algebra for $\Z_p^*$, or more directly through the system of rational linear equations implicit in  (\ref{Mazurintegral}) and  (\ref{GMPSMellin}) that define $\mu_k^*$ as nearly inverse to the rational-valued $\mu_{k,c}$.

\subsection{The Iwasawa algebra}
\label{sec:sub:Iwasawaalg}

Recall that $G_{\infty} \cong \Z_p^*\cong \Delta \times \Gamma$,   where $\Gamma= 1+p\Z_p$   and   $\Delta \cong G_1\cong (\Z/p\Z)^*$   is a cyclic group of order $p-1$.  The subgroup $\G$ is isomorphic to $\Z_p$ and in particular is the inverse limit $\displaystyle  \G=\lim_{\longleftarrow}\G_n$, where $\G_n=\G/\G^{p^n}$ and $\G^{p^n}=1+p^{n+1}\Z_p$.  As before, $\g$ denotes a fixed topological generator of $\G$, corresponding to $1\in \Z_p$ under this isomorphism.

One can think of $G_\infty$ as the Galois group $\operatorname{Gal}(K_\infty/\Q)$, where $K_\infty=\Q(\mu_{p^\infty})$ and $\mu_{p^\infty}$  denotes the group of $p$-power roots of unity. The isomorphism to $\Z_p^*$ is defined by sending $g\in \operatorname{Gal}(K_\infty/\Q)$ to the unique $\alpha \in \Z_p^*$ such that $g(\zeta)=\zeta^\alpha$ for all $\zeta\in \mu_{p^\infty}$.  With
this isomorphism, $\Gamma$ is identified with $\operatorname{Gal}(K_\infty/K)$  where $K = \Q(\mu_p)$ and $\mu_p$ is the group of $p$-th roots of unity.  Then $\g(\zeta)=\zeta^\kappa$ for a certain $\kappa\in 1+p\Z_p$ and all $\zeta\in \mu_{p^\infty}$.

 The {\em Iwasawa algebra} of $G_{\infty}$
 is by definition
\begin{equation}\label{IwasawaalgdefR}
R  \ \ = \ \   \Z_p[[G_{\infty}]]  \ \ = \ \   \lim_{\longleftarrow} \Z_p[G_n]\,,
\end{equation}
where the inverse limit is again  defined using the maps $\rho_{n}$.
Similarly, the  Iwasawa algebra of $\G$ is defined as the inverse limit
\begin{equation}\label{IwasawaalgdefLambda}
\L \ \ = \ \   \Z_p[[\G]]  \ \ = \ \   \lim_{\longleftarrow} \Z_p[\G_n]\,;
\end{equation}
it is isomorphic to the ring of formal Laurent series $\Z_p[[T]]$ via the correspondence  $\g\mapsto 1+T$ \cite[\S7]{washington}.
 The inverse isomorphism sends an element $F(T) \in \Z_p[[T]]$   to $\Theta=F(\gamma-id_{\Gamma})\in \Lambda$.  One can make sense of this by showing that the infinite series $F(\gamma-id_{\Gamma})$ converges in  $\Lambda$, which  is a local ring endowed with the  ${\mathfrak{m}}$-adic topology, where the maximal ideal ${\mathfrak{m}} = (p, \gamma-id_{\Gamma})$.

Thus $R= \Lambda[\Delta]$, the group ring for $\Delta$ over the ring $\Lambda$. Using the
$p-1$ idempotents $e_{\omega^i}$ for $\Delta$, one can write the compact topological ring $R$ as the direct sum of the $p-1$ ideals $e_{\omega^i}R$, each of which is a ring canonically  isomorphic to $\Lambda$.    Moreover, the ring of fractions ${\mathcal F}$ of $R$ is isomorphic to the direct sum of   the fraction fields of these rings: ${\mathcal F} \cong {\mathcal L}^{p-1}$,  where ${\mathcal L}$ is the fraction field of $\Lambda\cong \Z_p[[T]]$.  Thus an element of $\cal F$ is determined by its $p-1$ idempotent projections, which can be identified with ratios $F_{(i,1)}(T)/F_{(i,2)}(T)$, where $F_{(i,1)}(T)$ and   $F_{(i,2)}(T)$ are relatively prime elements of $\Z_p[[T]]$.

In Section~\ref{sec:sub:idempotents} we considered $\Q_p$-valued distributions on $\Z_p^*$ as elements $\Theta = (\theta_n)$ of $\displaystyle \lim_{\longleftarrow}\Q_p[G_n]$.  Elements   of $R$ are of course thus examples of $p$-adic measures on $\Z_p^*$, and elements of $\cal F$ are simply quotients $\Psi=\Theta_1/\Theta_2$ of $\Theta_1=(\theta_n^{(1)})$, $\Theta_2=(\theta_n^{(2)})\in R$ for which $e_{\omega^i}\Theta_2$   is nonzero for each $i=0,\ldots,p-2$.  This last condition is met, for example, if $\theta_n^{(2)}$ is invertible in the group algebra $\Q_p[G_n]$ for all $n$, or satisfies the equivalent requirement that $\chi(\theta_n^{(2)})\neq 0$ for all characters $\chi$ of $G_n$ (see (\ref{extendchi})).  Thus the quotient $\Psi$ can be thought of as the sequence $(\psi_n)$, where $\psi_n=\theta_n^{(1)}/\theta_n^{(2)}$; for characters $\chi$ of $G_n$ we may then define
$
\chi(\psi_n) = \chi(\theta_n^{(1)})\chi(\theta_n^{(2)})^{-1}$.

Any finite-order, continuous character $\chi$ of $G_\infty$ factors through $G_n$ for some $n$ and takes values in $\Z_p[\zeta]$, for some $p^t$-th root of unity $\zeta$.  Using formula (\ref{extendchi}), $\chi$ extends first to a ring homomorphism $\Z_p[G_m]\rightarrow \Z_p[\zeta]$, and then to a ring homomorphism $R\to \Z_p[\zeta]$ using the map $R\to \Z_p[G_m]$ implicit  in the profinite limit (\ref{IwasawaalgdefR}).
As we saw in section~\ref{sec:sub:idempotents}, $\Theta$ corresponds to a $\Z_p$-valued distribution $\nu$ on $G_{\infty}=\Z_p^{*}$   such that
\begin{equation}\label{integrationwithfiniteorderchar}
\chi(\Theta)   \ \ = \ \   \int_{\Z_p^{*}}~\chi \,d\nu\,.
\end{equation}
Recall that $R=\oplus_{0\le j < p-1}e_{\omega^j}R$.  Since $|\Delta|=p-1$  is relatively prime to $p$ and $G_\infty\cong \Z_p^*\cong \Delta\times\Gamma$ contains the $(p-1)$-st roots of unity, each idempotent $e_{\omega^j}\in \Z_p[\Delta]\subset \Z_p[[G_\infty]]=R$.  In particular,
\begin{equation}\label{chionidempotent}
    \chi(e_{\omega^j}) \ \ = \ \ \left\{
                                   \begin{array}{ll}
                                     1, & i=j\,; \\
                                     0, & i\neq j\,,
                                   \end{array}
                                 \right.
\end{equation}
where we  write $\chi=\omega^i\chi_{_\Gamma}$, with   $0\le i < p-1$ and $\chi_{_\Gamma}=\chi|_\Gamma$.  The restriction $\chi_{_\Gamma}$ also determines a ring homomorphism $\Lambda\to \Z_p[\zeta]$ using (\ref{extendchi}).
Then  $\chi(\Theta) = \chi(e_{\omega^i}\Theta)$, so
  $\chi(\Theta)$  depends only on the projection of $\Theta$  to the direct summand $e_{\omega^i}R$, a ring which is isomorphic to $\Lambda$ and on which the group homomorphisms induced by $\chi$ and $\chi_{_\Gamma}$ coincide.  Moreover, for  $\Psi=\Theta_1/\Theta_2$ in ${\mathcal F}$ with  $\chi(\Theta_2) \neq 0$,  $\chi(\Psi)=\chi(e_{\omega^i}\Psi)=\chi_{_\Gamma}(e_{\omega^i}\Psi)$.

\subsection{The Iwasawa element and its inverse}
\label{sec:sub:psiIw}

Recall that $k$ is a fixed integer at least 2.
There exists an element $\Psi_{Iw} = \{\psi_{Iw,m}\}$ in ${\mathcal F}$ such that
\begin{equation}\label{Bchidef1}
\chi(\psi_{Iw,m}) \ \ = \ \ b_\chi \ \ := \ \
\left\{
  \begin{array}{ll}
   (1-\chi(p)p^{k-1})\,L(1-k, \chi), & \qquad \chi(-1)=(-1)^k\,; \\
    1, & \qquad \chi(-1)\neq(-1)^k
  \end{array}
\right.
\end{equation}
for all $\chi$ factoring through $G_m$.
Its inverse, $\Psi_{Iw}\i\in \cal F$, satisfies the corresponding property
\begin{equation}\label{Bchidef2}
\chi\(\Psi_{Iw}\i \) \ \ = \ \  b_\chi\i
\end{equation}
for all finite order characters $\chi$,
and was the main object of study in \cite{GMPS} (it corresponds to the $p$-adic distribution $\mu_k^*$; see (\ref{example3A})).

One can define $\Psi_{Iw}$ by defining its projection to each component in the direct sum decomposition of ${\mathcal F}$ as $p-1$ copies of $\cal L$. It is a nontrivial element of $\Lambda$ in  $\f{p-1}{2} - 1$ components, has a nontrivial denominator in one component, and is just the constant 1 in the remaining $(p-1)/2$ components. The component where $\Psi_{Iw}$ has a nontrivial denominator corresponds to the existence of a simple pole at $s=1$ for $L_p(s, \omega^0)$.    Setting $t=0$ in (\ref{RGKubotaL}), we have
 \begin{equation}\label{RGconnectionwithLp}
    L_p(1-k,  \omega^0 \phi)  \ \  = \ \  (1- \omega^{-k}\phi(p)p^{k-1}) \,L(1-k, \omega^{-k}\phi)
\end{equation}
  for all finite order characters $\phi$ of $\Gamma$.    Thus,  the value of $i$ such that the projection to the $\omega^i$-component has a nontrivial denominator is determined by the congruence $i \equiv -k \pmod{p-1}$. One can take that denominator to be $\gamma - \kappa^k$, where $\gamma \in \Gamma$ and $\kappa \in 1+p\Z_p$ are as defined at the beginning of Section~\ref{sec:sub:Iwasawaalg}.   This denominator corresponds to the element $T- (\kappa^k-1)$ in $\Z_p[[T]]$.

\section{Proof of the upper bound}\label{sec:upperbd}

This section contains some standard results in basic Iwasawa theory, and is included for the convenience of readers unfamiliar with the topic.   In particular, the estimate of Proposition~\ref{RGProp1} (when applied to $\Psi_{Iw}$) shows that $|L(1-k,\chi)|$ is bounded above and below when $\chi$ ranges over Dirichlet conductors of $p$-power conductor (for $k$ and $p$ fixed), a result implicit in Iwasawa's papers.

Fix an integer $k\ge 1$ and a  prime $p$ (the assumption that $p$ is irregular is not necessary in this section, but since $\mu_k^*$ is actually bounded for regular primes $p$ the estimate one gets is far from the truth).  In this section we will show the existence of a constant $c_2$ such that
\begin{equation}\label{upperbd}
|  \mu^*_k(b + p^m \Z_p) \big|   \ \  \le   \ \   c_2\,p^m
\end{equation}
for all $b\in (\Z/p^m\Z)^*$ and $m\ge 1$.   Our argument is based on the following proposition.

\bigskip

\begin{proposition}\label{RGProp1}   Assume that $\Theta_1, \Theta_2\in R=\Z_p[[G_\infty]]$ have $\chi(\Theta_1)\neq 0$ and $\chi(\Theta_2)\neq 0$ for all finite order characters $\chi$ of $G_\infty\cong \Z_p^*$.   Let $\Psi=\Theta_1/\Theta_2$.   Then $|\chi(\Psi)|$ is bounded above and below independently of the finite order character $\chi$.
\end{proposition}

\bigskip

\noindent   {\bf Remark.}     The assumptions in the proposition imply that $\Psi$ is defined and furthermore that $\chi(\Psi)$ is defined and is nonzero for all finite order characters  $\chi$ of $G_{\infty}$.
\medskip

\noindent The proof depends on two lemmas. We let ${\mathcal {Z}}$ denote the set of $p$-power roots of unity in $\overline{\Q}_p$.

\bigskip

\begin{lemma}\label{RGLemma1}
   Suppose that $F(T)$ is a nonzero element of the formal power series ring $\Z_p[[T]]$ and that $F(\zeta-1) \neq 0$  for all $\zeta \in {\mathcal {Z}}$.   Then there exists a positive real number $c$ such that
\begin{equation}\label{upbd4}
c \ \ \le \ \ |F(\zeta-1)| \ \ \le  \ \ 1  \ \ \ \ \ \  \text{for all~}\zeta\, \in \, {\mathcal Z}\,.
\end{equation}
\end{lemma}

\medskip

\noindent
The  proof below follows an argument which Iwasawa gave in order to prove his formula   for the power of $p$ dividing the class numbers of the layers in a $\Z_p$-extension of a number field \cite[pp.~92-93]{Iwabook}.  It furthermore precisely determines $\big|F(\zeta-1)\big|$  when $\zeta \in {\mathcal {Z}}$ has sufficiently high order.

 \bigskip

\begin{proof}  The power series $F(T)$  converges at $T=\zeta-1$ because $|\zeta-1| < 1$.    Since $\Z_p$ is closed, we have $F(\zeta-1) \in \Z_p$. This gives the upper bound.
It suffices to establish the lower bound when $\zeta$ is a primitive root of unity of sufficiently high order (since this omits all but a finite number of values, all of which are nonzero).    Without loss of generality, we can divide $F(T)$ by a suitable power of $p$ so that at least one coefficient is in $\Z_p^*$.  That is, we may assume that
  $F(T)= \sum_{i=0}^{\infty}  b_iT^i \in \Z_p[[T]]$ and that
\begin{equation}\label{upbd5}
\big| b_i \big| ~\le ~ p^{-1} ~~~\text{for} ~~i < \lambda\,,  \quad \quad \quad \big| b_{\lambda} \big| ~=~ 1\,,   \quad \quad \quad    \big| b_i \big| ~ \le ~1 ~~~for ~~i ~> ~ \lambda\,.
\end{equation}
for some $\lambda \ge 0$.

Suppose that  $\zeta \in {\mathcal {Z}}$ and $\zeta \neq 1$.  The order of $\zeta$ is $p^t$   for some $t  \ge 1$.   Then
$\zeta-1$ generates the maximal ideal in $\Q_p(\zeta)$, which is a totally ramified extension of $\Q_p$ of degree $(p-1)p^{t-1}$.  Consequently,
\begin{equation}\label{upbd6}
\big| \zeta-1 \big|  \ \ = \ \    p^{-\frac{1}{(p-1)p^{t-1}}} \ \ < \ \ 1\,.
\end{equation}
Note that the left hand side tends to 1 as $t\rightarrow\infty$.

Now we have
\begin{equation}\label{upbd7}
\big| b_{i} (\zeta-1)^i \big|   \ \ = \ \   \big| b_i \big|  \cdot   \Big(p^{-\frac{1}{(p-1)p^{t-1}}} \Big)^i
\end{equation}
for $i \ge 0$.  In the special case $i=\lambda$  this absolute value is   $\Big(p^{-\frac{1}{(p-1)p^{t-1}}}\Big)^{\lambda}$, which approaches 1 as $t\rightarrow \infty$.
Thus for $t$ sufficiently large   we have
\begin{equation}\label{upbd8}
\big| b_{\lambda} (\zeta-1)^{\lambda} \big|  \ \  >  \ \  p^{-1} ~~ \ge ~~ \big| b_{i} (\zeta-1)^i \big| \  \ \ \ \ \ \ \text{for} \ i\,<\,\l\,.
\end{equation}
We also have
\begin{equation}\label{upbd9}
\big| b_{\lambda} (\zeta-1)^{\lambda} \big|  ~=~    \Big(p^{-\frac{1}{(p-1)p^{t-1}}}\Big)^{\lambda}~>~ \Big(p^{-\frac{1}{(p-1)p^{t-1}}}\Big)^{\lambda + 1}  ~\ge ~   \big| b_{i} (\zeta-1)^i \big|
\end{equation}
for all $i \ge \lambda+1$.    Thus, when  $t$ is sufficiently large,  the term for $i=\lambda$ in (\ref{upbd7}) strictly dominates all the other terms in absolute value and consequently
\begin{equation}\label{upbd10}
\big| F(\zeta-1) \big|   ~~=~~  \big| b_{\lambda} (\zeta-1)^{\lambda} \big|  ~~=~~ \Big(p^{-\frac{1}{(p-1)p^{t-1}}}\Big)^{\lambda}~~=~~  p^{-\frac{\lambda}{(p-1)p^{t-1}}}.
\end{equation}
This last quantity tends to 1 as $t\rightarrow \infty$, completing the proof.\bx
\end{proof}

\begin{lemma}\label{RGLemma2}  Let $F(T)$ be a nonzero element of the fraction field of $\Z_p[[T]]$ for which $F(\zeta-1)$ is well-defined and nonvanishing for all $\zeta\in {\mathcal Z}$.  Then $\{|F(\zeta-1)| ~~\big| ~\zeta\in {\mathcal Z}\}$ is bounded above and below.
\end{lemma}

\begin{proof}   Write $F(T)=F_1(T)/F_2(T)$, where $F_1(T),  F_2(T)\in \Z_p[[T]]$ are relatively prime.  Since $F_1(T)$ and $F_2(T)$ are relatively prime, they cannot simultaneously vanish on an element of $\overline{\Q}_p$.  Thus, by assumption on $F(T)$, neither power series vanishes at $\zeta-1$ for $\zeta\in \mathcal Z$.  The claim now follows from   (\ref{upbd4}). \bx
\end{proof}

\noindent To prove  Proposition~\ref{RGProp1},  we argue as follows.     Choose $F_{(i,1)}(T)$, $F_{(i,2)}(T)\in \Z_p[[T]]$ such that the projections $e_{\omega^i}\Theta_j$ correspond to $F_{(i,j)}(T)\in \Z_p[[T]]$ under the isomorphism   $\Lambda=\Z_p[[\Gamma]]\cong \Z_p[[T]]$ described after (\ref{IwasawaalgdefLambda}). Recall that  $T$ is identified with $\gamma-id_{\Gamma}$ in that isomorphism.  If $\chi$ is a finite order character of $\Z_p^*$, then $\zeta=\chi(\gamma)$ is necessarily a $p^t$-th root of unity  for some $t\ge 0$.   Furthermore, $\chi=\chi_{_\Gamma} \omega^i$ for some  $0 \le i < p-1$ as at the end of Section~\ref{sec:sub:valuations}.   For that value of $i$, we have
\begin{equation}\label{chionLaurent}
     \chi(\Theta_j) \ \ = \ \    F_{(i,j)}(\zeta-1)
\end{equation}
for $j=1,2$.    We have assumed these quantities do not vanish for any $i$:
\begin{equation}\label{Fisnonvanishing}
    F_{(i,1)}(\zeta-1)\,, \  F_{(i,2)}(\zeta-1) \ \ \neq \ \ 0 \ \ \ \ \ \text{for} \ \  \ 0\,\le\,i\,<\,p-1\,.
\end{equation}
  Furthermore varying $\chi$ over all finite order characters, the nonvanishing statement (\ref{Fisnonvanishing}) holds for any $\zeta\in \cal Z$.  Thus Lemma \ref{RGLemma2} implies that $|F_{(i,1)}(T)/F_{(i,2)}(T)|$ is bounded above and below, from which it follows that $|\chi(\Psi)|$ is as well.  This proves Proposition~\ref{RGProp1}.
\bigskip

Finally, we now deduce  the inequality  (\ref{upperbd}) from Proposition~\ref{RGProp1}.  Suppose $\Psi$ is as in the proposition, so that there exists some $d>0$ with   $|\chi(\Psi)| < d$ for all finite order characters $\chi$ of $\Z_p^*$.
As in Section~\ref{sec:sub:idempotents}, we may
 obtain a distribution $\nu$    on $\Z_p^*$ represented by the compatible sequence $\{\psi_m\}$ associated to $\Psi$ by writing $\psi_m = \sum_{g \in G_m} ~a_g g$, where
\begin{equation}\label{upbd2}
a_g  ~=~     \frac{1}{\#G_m}    \sum_{\chi} ~ \chi^{-1}(g)\chi(\psi_m) \ \ = \ \
\!\nu(b+p^m\Z_p)
\end{equation}
for $g = b+p^m\Z_p$, an element of $G_m$ when $p \nmid b$.
Since the roots of unity $\chi(g)$ have valuation 1 and $\#G_m = (p-1)p^{m-1}$, the ultrametric inequality yields the bound
\begin{equation}\label{upbd3}
\big|\nu(b+ p^m \Z_p) \big| \  \ = \ \ |a_g |  \ \  \le   \ \   d \,|  \#G_m  |^{-1}   \ \ = \ \   c_2 p^m
\end{equation}
for some positive real  number $c_2$.
We obtain the inequality  (\ref{upperbd}) from   (\ref{upbd3})  by taking $\Psi =  \Psi_{Iw}^{-1}$.  For this choice of  $\Psi$,  the assumptions in Proposition~\ref{RGProp1} are clearly satisfied and, by definition, we have $\nu = \mu_k^*$.

\section{Proof of the lower bound}\label{sec:lowerbd}

In this section we will prove that
\begin{equation}\label{freydoon1}
\max_{b\,\in\,(\Z/p^m\Z)^*} |\mu_k^*(b+p^m\Z_p)|_p \ \ \geq \ \  c_1\,p^m \ ,  \ \ \ \ \ m\,\ge \,1\,,
\end{equation}
for some positive constant $c_1$ which depends only on the irregular prime $p$.
The distribution $\mu_k^*$ corresponds to the element $\Theta= \Psi_{Iw} ^{-1}\in \mathcal F$ as in section~\ref{sec:sub:psiIw}.
The definition of $\mu_k^*$ shows that the maximum is nonzero, and so it suffices to prove (\ref{freydoon1}) for sufficiently large $m$.

Recall that the distribution $\mu_k^*$ corresponds to a compatible sequence of elements in the rings $\Q_p[G_m]$, namely the sequence $\{\psi_{Iw,m}^{-1}\}$.  We shall continue to denote
  $\psi_{Iw,m}^{-1}$ by $\xi_m$ for brevity.
For every character $\chi$ of $G_m$ and every $m\ge 1$, equation (\ref{Bchidef2}) shows that $\chi(\xi_m)=b_\chi^{-1}$,  which is a nonzero number by definition (\ref{Bchidef1}).  This nonvanishing is essential for defining $\Psi\in {\mathcal F}$ (see Section~\ref{sec:sub:Iwasawaalg}).

The ring ${\mathcal F}$ is a direct sum of $p-1$ copies of ${\mathcal L}$, the field of fractions of the ring $\Lambda=\Z_p[[\Gamma]]$, where $\Gamma=1+p\Z_p$.
The ring $\Lambda$ is isomorphic to the formal power series ring $\Z_p[[T]]$ and is a unique factorization domain.
The decomposition corresponds to the $p-1$ idempotents $e_{\omega^i}$ for the group $\Delta$, where $0\leq i\leq p-2$.
Thus, $\Psi_{Iw}\i$ has $p-1$ projections in copies of ${\mathcal L}$.
As in the proof of Lemma~\ref{RGLemma2} in Section~\ref{sec:upperbd}, we denote these projections by $F_{(i,1)}(T)/F_{(i,2)}(T)$, where naturally
  $F_{(i,1)}(T)$ and $F_{(i,2)}(T)$ are taken to be relatively prime elements of $\Z_p[[T]]$.

Recall from Section~\ref{sec:sub:psiIw}  that the element $\Psi_{Iw}$ of ${\mathcal F}$ gives the $p$-adic $L$-function attached to $\chi$ by (\ref{Bchidef1}) as $\chi(\Psi_{Iw})$.
The Ferrero--Washington  theorem \cite{FerrWash} applies to our situation since we are assuming that $p$ is irregular (as we must, for otherwise the inequality (\ref{freydoon1}) is false).  It asserts that for irregular primes $p$, there exists an index $i$ such that $F_{(i,2)}/F_{(i,1)}$  --- the element of $\mathcal L$ corresponding to the $i$-th projection of $\Psi_{Iw}$ ---  has a zero in some field extension.  By the $p$-adic Weierstrass preparation theorem \cite[\S7.1]{washington}, this $F_{(i,2)}(T)$ can be factored as
$
F_{(i,2)}(T)=p^a u(T) g(T)
$, where
 $u(T)$ is an invertible element in $\Lambda$, $a\geq 0$ an integer, and $g(T)$ is a polynomial of strictly positive degree having a root $\beta$ in the maximal ideal $\frak m$
of the ring of integers $\mathcal O$ of some finite extension of $\Q_p$ (i.e., $|\beta| < 1$).
As $F_{(i,1)}(T)$ and $F_{(i,2)}(T)$ are relatively prime, $F_{(i,1)}(\beta)\neq 0$.

Since $1+\beta$ is an element of the subgroup $1+\frak m$  of the unit group ${\mathcal O}^\times$ of $\mathcal O$, we may define a homomorphism
\begin{equation}\label{RGphibetas}
\aligned
\varphi \ & \colon \  \Gamma \ \ \to  \ \ 1+\frak m \\
\varphi \ & \colon \ \gamma^{\,s} \ \ \mapsto \ \ (1+\beta)^s\,,
\endaligned
\end{equation}
where $\gamma$ is the topological generator of $\Gamma$ and $s$ varies over $\Z_p$; it is continuous since $1+\frak m$ is a pro--$p$--group.  Furthermore, for $i$ as above consider the homomorphism
$$
\chi\colon G_\infty\to{\mathcal O}^\times
$$
given by $\chi=\omega^i\varphi$, which
  makes sense because $G_\infty=\Delta\times\Gamma$ and $\omega^i$ is a character of $\Delta$ with values in $\Z^\times_p\subseteq{\mathcal O}^\times$.
(Note that $\varphi$ and $\chi$ may not necessarily have finite order, because $1+\beta$ may not be a root of unity.)
As before, formula (\ref{extendchi}) gives an extension of $\chi$ to a continuous $\Z_p$--algebra homomorphism:\  $\Z_p[[G_\infty]]\to{\mathcal O}$.

We now specialize the above  discussion to $\Psi_{Iw}\i$, which we write as the quotient $\Theta_1/\Theta_2$, where $\Theta_1$, $\Theta_2\in R$ satisfy
\begin{equation}\label{eqn:chi(Theta)}
\chi(\Theta_2) ~=~ 0 \quad \quad \quad \text{and} \quad \quad \quad \chi(\Theta_1) ~ \neq ~0 ~~.
\end{equation}
Both  $\Theta_1$ and $\Theta_2$ correspond to compatible sequences $\{\theta_m^{(1)}\}$ and $\{\theta_m^{(2)}\}$ of elements in $\Z_p[[G_m]]$ for varying $m$. Also, $\Psi_{Iw}\i=\Theta_1 \big/ \Theta_2$ corresponds to a compatible sequence $\{\psi_{Iw,m}\i\}$, where $\psi_{Iw,m}\i=\theta_m^{(1)}/\theta_m^{(2)}$.    We will now interpret (\ref{eqn:chi(Theta)}) in terms of these sequences.

The coefficients of $\psi_{Iw,m}^{-1}$ are the numbers $\mu^*_k(b + p^m \Z_p)\in \Q_p$, for $b\in (\Z/p^m\Z)^*$.  It is thus clear that the numbers $t_m$ defined by
\begin{equation}\label{ptmdef}
p^{t_m} ~=~
\max_{\srel{ 1 \le b \le p^m}{p \, \nmid \, b}}\( \big|  \mu^*_k(b + p^m \Z_p) \big|_p \)
\end{equation}
are integers for which
\begin{equation}\label{ptmpsimintegral}
\tau_m \ \ := \ \ p^{t_m}\,\xi_{m}   \ \ \in  \ \ \Z_p[G_m]\,.
\end{equation}
 Let $T_m$ be any lifting of $\tau_m$ to the ring $\Z_p[[G_{\infty}]]$; thus $T_m$ maps to $\tau_m$ under the natural map from $\Z_p[[G_{\infty}]]$ to $\Z_p[G_m]$.    The kernel of the group homomorphism  $G_{\infty} \to G_m$ is topologically generated  by $\gamma^{p^{m-1}}$, and the kernel of the ring homomorphism  $\Z_p[[G_{\infty}]]  \to  \Z_p[G_m]$ is the ideal in $\Z_p[[G_{\infty}]]$ generated by $\gamma^{p^{m-1}}-1$.

By definition, we have the equality
\[
\theta_m^{(2)} \tau_m ~=~   p^{t_m}\theta_m^{(1)}
\]
in the ring $\Z_p[G_m]$,  and therefore the following congruence in the ring $\Z_p[[G_{\infty}]]$:
\[
\Theta_2 \, T_m  ~ \equiv  ~ p^{t_m} \Theta_1 ~~ \pmod{ (\gamma^{p^{m-1}}-1)}\,.
\]
Consequently $p^{t_m}\Theta_1  =  \Theta_2 T_m  +   \eta_m \big( \gamma^{p^{m-1}}-1\big)$ for some $\eta_m \in \Z_p[[G_{\infty}]]$. Now apply $\chi$ to this equation and use the fact that $\chi(\Theta_2)=0$ to obtain
\[
p^{t_m} \chi(\Theta_1)   ~ \equiv ~ 0 ~ \pmod{{(1+\beta)^{p^{m-1}} -1}}
\]
in the ring ${\mathcal O}$.  Since $\chi(\Theta_1)$ is nonzero and independent of $m$, the valuation of the left side is $p^{-t_m -c}$ for some constant $c\in \Z$. The valuation of $(1+\beta)^{p^{m-1}}-1$ is
\begin{equation}\label{binomialexpansion}
\aligned
   |(1+\beta)^{p^{m-1}}\,-\,1| \ \ & = \ \ \left| \sum_{j=1}^{p^{m-1}}{{p^{m-1}}\choose{j}} \beta^j \right|\\
& \le \ \   \max_{1\,\le\,j\,\le\,p^{m-1}} \left|{{p^{m-1}}\choose{j}} \beta^j\right| \ \  = \ \   |p^{m-1}| \max_{1\,\le\,j\,\le\,p^{m-1}} \left|\f{\beta^j}{j}\right|\,,
\endaligned
\end{equation}
 which is less than a constant times $p^{-m}$ since $|\beta|<1$ and the exponent of the highest power of $p$ dividing $j$ grows at most logarithmically in $j$.   Thus $t_m-m$ is bounded below by a constant, as was to be shown.

\bigskip

\noindent
{\bf Remark:}
The argument given here applies more generally to the distribution $\nu$ associated to an element $\Psi=\Theta_1/\Theta_2$, where $\Theta_1,\Theta_2\in R$ satisfy the following assumptions.  First, we must assume that $\chi(\Theta_2)\neq 0$ for all finite order characters $\chi$ of $\Z_p^*$.  This is needed in order to define the distribution $\nu$.  In addition, we assume the existence of a continuous homomorphism $\chi:\Z_p^*\to\overline{\Q_p}^*$ such that $\chi(\Theta_2)=0$ and $\chi(\Theta_1)\neq 0$.  In terms of the projections  to the various $p-1$ copies of $\Lambda$ in the direct sum decomposition of $R$, this assumption amounts to the following requirement on the relatively-prime elements $F_{i,1}(T)$ and $F_{i,2}(T)\in \Z_p[[T]]$ corresponding to the images of $\Theta_1$ and $\Theta_2$:~there must exist at least one value of $i$ for which $F_{i,2}(T)$ can be written as $p^{\mu_i}G_i(T)$, where $\mu_i\in\Z$, $G_i(T)\in \Z_p[[T]]$ has at least one coefficient in $\Z_p^*$, but $G_i(0)\notin\Z_p^*$.

\bibliographystyle{plain}

\end{document}